\documentclass[a4paper, 12pt]{article}
\usepackage{amsthm}
\usepackage{amsmath}
\usepackage{amsfonts}
\usepackage{amssymb}
\usepackage{mathrsfs}
\usepackage{url}
\usepackage[colorlinks=true]{hyperref} 
\hypersetup{urlcolor=blue,linkcolor=black,citecolor=black,colorlinks=true}
\usepackage{enumitem}
\usepackage[top=4cm, bottom=4cm, left=3.5cm, right=3.5cm]{geometry}

\theoremstyle{definition}

\newtheorem{Lemma}{Lemma}
\newtheorem{Theorem}{Theorem}

\DeclareMathOperator{\Wron}{Wron}

\newcommand{\SLZ}{\operatorname{SL}_2(\mathbb{Z})}
\newcommand{\GL}{\operatorname{GL}}
\newcommand{\Mod}[1]{\ (\text{mod}\ #1)}

\newcommand{\bC}{\mathbb{C}}

\newcommand{\bZ}{\mathbb{Z}}

\newcommand{\cH}{\mathcal{H}}

\newcommand{\fg}{\mathfrak{g}}

\newcommand{\fu}{\mathfrak{u}}

\usepackage{fancyhdr}

\title{ 
\bf An extension of \\ Macdonald's identity for $\mathfrak{sl}_{n}$\footnote{This is a pre-print of an article published in Research in Number Theory. The final authenticated version is available online at: https://doi.org/10.1007/s40993-019-0163-0}\\[.3em] 
\small Quentin Gazda\footnote{Current adress: 
Univ Lyon, Universit\'e Jean Monnet Saint-\'Etienne, CNRS UMR 5208, Institut Camille Jordan, F-42023 Saint-\'Etienne, France}}
\author{}
\date{June, 2019}

\begin{document}
\maketitle

\begin{abstract}
Let $n$ be an odd positive integer. In this short elementary note, we slightly extend Macdonald's identity for $\mathfrak{sl}_{n}$ into a two-variables identity in the spirit of Jacobi forms. The peculiarity of this work lies in its proof which uses Wronskians of vector-valued $\theta$-functions. This complements the work of A. Milas towards modular Wronskians and denominator identities. 
\end{abstract}

Let $\eta$ denote the Dedekind $\eta$-function, given by the infinite product
\begin{equation}
\eta(\tau)\stackrel{\text{def}}{=}q^{\frac{1}{24}}\prod_{n=1}^{\infty}{(1-q^n)} \nonumber
\end{equation}
where $\tau\in \cH=\{z\in \bC|\operatorname{im}(z)>0\}$ is the Poincar\'e upper-half plane, and $q$ is the local parameter $e^{2i\pi \tau}$ at infinity. On the other hand, let $R$ be a reduced root system in a real vector space $V$ canonically attached to a semi-simple Lie algebra $\fg$ over $\bC$. Let $(\cdot,\cdot)$ be a scalar product on $V$ invariant under the action of the Weyl group of $R$. To settle notations, we let $\|x\|^2=(x,x)$ for $x\in V$, we fix $\Phi$ the highest root of $R$ and set $g=(\Phi,\rho)$ for $\rho$ half the sum of the positive roots in $R$ (\textit{positive} being defined with respect to the early choice of a Weyl chamber). We let $\Lambda$ be the lattice in $V$ generated by the set $\{2g\alpha/\|\alpha\|~|~\alpha \in R\}$. In the celebrated paper \cite{macdonald}, I. G. Macdonald proved a remarkable identity for the $\dim \fg$-power of the Dedekind $\eta$-function:
\begin{equation}\label{macdonald}
\eta(\tau)^{\dim \fg}=\sum_{l\in \Lambda}{\prod_{\alpha>0}{\frac{(l+\rho,\alpha)}{(l,\alpha)}}~q^{\frac{\|l+\rho\|^2}{2g}}} \quad (\tau\in \cH).
\end{equation}
For $n$ an odd positive integer and $\fg=\mathfrak{sl}_{n}$, we draw from  \cite[App.~1(6)(a)]{macdonald} that equation \eqref{macdonald} reduces to
\begin{equation}\label{sln}
 \eta(\tau)^{n^2-1}=\frac{1}{1!2!\cdots (n-1)!}\sum_{\substack{(x_1,...,x_{n})\in \bZ^{n} \\ x_i\equiv~ i~(\text{mod}~n) \\ x_1+...+x_{n}=0}}{\prod_{i<j}{(x_i-x_j)}~q^{\frac{1}{2n}(x_1^2+...+x_{n}^2)}}.
\end{equation}
For $n=5$, \eqref{sln} is known as Dyson's identity and is equivalent to an astounding formula for the Ramanujan function  $\tau$ (presented in \cite{dyson}).

For $z\in \bC$, we let $\zeta=e^{2i\pi z}$. Let $\theta$ be the usual theta function:
\begin{equation}
\theta(\tau,z)\stackrel{\text{def}}{=} \sum_{r\in \bZ}{q^{\frac{r^2}{2}}\zeta^r}=\prod_{m=1}^{\infty}{(1-q^m)(1+q^{\frac{2m-1}{2}}\zeta)(1+q^{\frac{2m-1}{2}}\zeta^{-1})} \nonumber
\end{equation}
where we wrote the \textit{Jacobi triple product} on the right. In this short note, we extend Equation \eqref{sln} into a {\it two-variables} identity involving $\theta$. 
\begin{Theorem}\label{macdoextend}
Let $n$ be an odd positive integer. For all $\tau\in \cH$, $z\in\bC$, we have
\begin{equation}
\theta(\tau,z)\eta(\tau)^{n^2-1}=\frac{1}{1!2!\cdots (n-1)!}\sum_{(x_1,...,x_n)}{\prod_{i<j}{(x_i-x_j)}~q^{\frac{(x_1^2+...+x_n^2)}{2n}}\zeta^{\frac{(x_1+...+x_n)}{n}}} \nonumber
\end{equation}
where the sum is over $n$-tuples $(x_1,...,x_n)$ in $\bZ^n$ such that $x_i\equiv i\Mod{n}$ for all $i$.
\end{Theorem}

In fact, Theorem \ref{macdoextend} is equivalent to \eqref{sln}, namely Macdonald's identity for $\mathfrak{sl}_n$ for odd $n$. One deduces \eqref{sln} simply by comparing the constant coefficients once considered as an equality of Fourier series in $z$, and Theorem \ref{macdoextend} follows by taking the sum over $r\in \mathbb{Z}$ of the right-hand side of \eqref{sln} whose summand is shifted by $(x_1,...,x_n)\mapsto (x_1+r,...,x_n+r)$. 

Our proof of Theorem \ref{macdoextend}, however, is entirely modular and mainly self-contained. In particular, it avoids the use of root system of Lie algebras. Essentially, the method presented here fits naturally in the framework of Jacobi forms: it consists of computing the Eichler-Zagier decomposition of the Wronskian of a family of theta functions. As such, the method employed here is in many ways reminiscent of the work of A. Milas in \cite{milas2}, \cite{milas3} where Wronskians are used to give alternative proofs of Macdonald's identities for root systems of type $B_n$, $C_n$, $BC_n$ and $D_n$. 

This note is designed to be elementary and the reader does not need any background in the theory of Jacobi forms.

\section{Proof of Theorem \ref{macdoextend}}
\label{sec:1}

We now turn to the proof of Theorem \ref{macdoextend}. For $n=1$, the result is clear, so we can assume that $n>1$. Let $\Gamma$ denote the following congruence subgroup of $\SLZ$ of index $3$:
\begin{equation}
\Gamma=\left\{\begin{pmatrix}a & b \\ c & d \end{pmatrix}\in \SLZ~\vert ~\begin{pmatrix}a & b \\ c & d \end{pmatrix}\equiv \begin{pmatrix}1 & 0 \\ 0 & 1 \end{pmatrix}~\text{or}~\begin{pmatrix}0 & 1 \\ 1 & 0 \end{pmatrix}\pmod 2 \right\}. \nonumber
\end{equation}
It is generated by the two elements $S=\left(\begin{smallmatrix} 0 & -1 \\ 1 & 0 \end{smallmatrix}\right)$ and $T^2=\left(\begin{smallmatrix} 1 & 2 \\ 0 & 1 \end{smallmatrix}\right)$.

As we will use them extensively, we introduce the classical \textit{slash operators} for modular and Jacobi forms. For all integer $k$, we define the set $G_{k|2}(\Gamma)$ consisting of pairs $\left[\sigma,\varphi\right]$ where $\sigma=\left(\begin{smallmatrix} a & b \\ c & d \end{smallmatrix}\right)\in \Gamma$ and where $\varphi:\mathcal{H}\to \bC$ is a holomorphic function such that there exists a root of unity $t$ for which $\varphi(\tau)^2=t^k(c\tau+d)^k$ (for all $\tau\in \mathcal{H}$). The function $\varphi$ will be referred as the automorphy factor (of weight $k/2$) of the pair. As usual, $j$ will be the automorphy factor $(\sigma,\tau)\mapsto c\tau+d$. For legibility, the dependency on $\sigma$ in the automorphy factor shall not appear. The set $G_{k|2}(\Gamma)$ is again a group according to the operation $\left[\sigma_1,\varphi_1\right]\left[\sigma_2,\varphi_2\right]=\left[\sigma_1\sigma_2,\tau\mapsto \varphi_1(\sigma_2\tau)\varphi_2(\tau)\right]$. Note that for $\sigma\in \Gamma$, if $\left[\sigma,\varphi_1\right]\in G_{k_1|2}(\Gamma)$ and $\left[\sigma,\varphi_2\right]\in G_{k_2|2}(\Gamma)$, then $\left[\sigma,\varphi_1\varphi_2\right]\in G_{k_1+k_2|2}(\Gamma)$. For a holomorphic function $f:\cH\to \bC$ and $[\sigma,\varphi]=\left[\left(\begin{smallmatrix} a & b \\ c & d \end{smallmatrix}\right),\varphi\right]\in G_{k|2}(\Gamma)$, let $(f|\left[\sigma,\varphi\right])$ be the $\bC$-valued holomorphic function given for all $\tau\in \cH$ by
\begin{equation}
(f|\left[\sigma,\varphi\right])(\tau)\stackrel{\text{def}}{=}\varphi(\tau)^{-1}f\left(\frac{a\tau+b}{c\tau+d}\right). \nonumber
\end{equation}
Let $m$ be a rational number. For a function $\phi:\mathcal{H}\times \bC\to \bC$, holomorphic in its two variables, we also let 
\begin{equation}\label{jacobislash}
(\phi|_{m}\left[\sigma,\varphi\right])(\tau,z)\stackrel{\text{def}}{=}\varphi(\tau)^{-1}\exp \left(-\frac{2i \pi mcz^2}{c\tau+d}\right)\phi\left(\frac{a\tau+b}{c\tau+d},\frac{z}{c\tau+d}\right).
\end{equation}
We extend coordinate-wise the slash operators on vectors of functions. The slash operator \eqref{jacobislash} was introduce to define Jacobi forms by Eichler and Zagier in \cite{eichlerzagier}.

As for the classical slash operator, \eqref{jacobislash} is not stable by derivative in the variable $z$. There is, however, a relation to the Wronskian of a vector of functions. For $\Phi:\cH\times \bC\longrightarrow \bC^n$, we define
\begin{equation}
\Wron_z \Phi\stackrel{\text{def}}{=}\det(\Phi,\partial_z\Phi,...,\partial_z^{n-1} \Phi). \nonumber
\end{equation}
From Leibniz' derivation rule, we have that $\partial_z^r(\phi|_m[\sigma,\varphi])=(\partial_z^r\phi)|_{m}[\sigma,j^r\varphi]+$ a linear combination of $(\partial_z^i\phi)|_{m}[\sigma,j^i\varphi]$ ($i<r$) with complex coefficients independent of $\phi$. By multilinearity of $\det$, it follows that  
\begin{equation}\label{wronz}
\Wron_z (\Phi|_m\left[\sigma,\varphi \right])=(\Wron_z \Phi)|_{mn}\left[\sigma,j^{\frac{n(n-1)}{2}}\varphi^n\right].
\end{equation}
If $[\sigma,\varphi]$ is in $G_{k|2}(\Gamma)$, $\left[\sigma,j^{\frac{n(n-1)}{2}}\varphi^n\right]$ belongs to $G_{K|2}(\Gamma)$ where $K=n(n+k-1)$. \\

Let $n$ be an odd positive integer, and consider  
\begin{equation}
\theta_{n,i}(\tau,z)\stackrel{\text{def}}{=}\sum_{\substack{x\in \bZ \\ x\equiv i\Mod{n}}}{q^{\frac{x^2}{2n}}\zeta^x} \quad (z\in \bC,~\tau\in \cH). \nonumber
\end{equation}
Let $\theta_n$ be the transpose of $(\theta_{n,1},...,\theta_{n,n})$. The following Lemma is deduced from Poisson's formula:
\begin{Lemma}\label{theta}
There exists a unitary representation $\fu_n:\Gamma\to \GL_n(\bC)$ such that, for all $\sigma\in \Gamma$, there exists $\varphi:\cH\to \bC$ for which $[\sigma,\varphi]\in G_{1|2}(\Gamma)$ and
\begin{equation}
\left(\theta_{n}|_{\frac{n}{2}}[\sigma,\varphi]\right)=\fu_n(\sigma)\theta_n. \nonumber
\end{equation}
\end{Lemma}
As Lemma \ref{theta} is classical, we leave it without proof (see \cite[Section~II.5]{eichlerzagier}). The representation $\fu_n$ and the automorphy factor $\varphi$ can be made explicit, but for our purpose we simply need to know that $\fu_n$ is unitary. Note that the function $\theta$ of the introduction is here denoted $\theta_{1,0}$. 

By \eqref{wronz} and for $[\sigma,\varphi]$ as in Lemma \ref{theta}, we find that
\begin{equation}\label{modularityhr}
(\Wron_z\theta_{n})|_{\frac{n^2}{2}}\left[\sigma,j^{\frac{n(n-1)}{2}}\varphi^n\right]=(\det \fu_n(\sigma)) (\Wron_z\theta_{n}).
\end{equation}
Vandermonde's identity enables us to compute its Fourier expansion:
\begin{Lemma}
For all $\tau\in \cH$ and $z\in \bC$, we have
\begin{equation}
(\Wron_z\theta_{n})(\tau,z)=\sum_{\substack{(x_1,...,x_n)\in \bZ^n \\ x_i\equiv~ i\pmod{n}}}{\prod_{i<j}{(x_i-x_j)}~q^{\frac{1}{2n}(x_1^2+...+x_n^2)}\zeta^{x_1+...+x_n}}. \nonumber
\end{equation}
\end{Lemma}
\begin{proof}
It is enough to prove the formula formally. We have
\begin{align*}
\Wron_z \theta_n&= \det (\theta_n,\partial_z \theta_n,...,\partial^{n-1}_z \theta_n) \\
&= \sum_{x_1,x_2,... \equiv 1,2,...\Mod{n}}{\det (1,(x_i)_i,...,(x_i^{n-1})_i)q^{\frac{1}{2n}(x_1^2+...+x_n^2)}\zeta^{x_1+...+x_n}}.
\end{align*}
By Vandermonde's identity, $\det (1,(x_i)_i,...,(x_i^{n-1})_i)=\prod_{i<j}{(x_i-x_j)}$.
\end{proof}
This Lemma implies that $(\Wron_z \theta_n)(\tau,z/n)$ is, up to the factor $1!2!\cdots (n-1)!$, the member on the right-hand side in Theorem \ref{macdoextend}. The function appearing in right-hand side of \eqref{sln} is rather related to $h_r$ that we now define. For $r\in \bZ$, $\tau\in \cH$, let
\begin{equation}\label{defhr}
h_r(\tau)\stackrel{\text{def}}{=}\sum_{\substack{(x_1,...,x_n)\in \bZ^n \\ x_i\equiv~ i~(\text{mod}~n)\\ x_1+...+x_n=r}}{\prod_{i<j}{(x_i-x_j)}~q^{\frac{1}{2n}(x_1^2+...+x_n^2)-\frac{r^2}{2n^2}}}. 
\end{equation}
Note that, for all $\tau\in \cH$ and $z\in \bC$,
\begin{equation}\label{wronhr}
(\Wron_z\theta_{n})(\tau,z)=\sum_{r\in \bZ}{h_r(\tau)q^{\frac{r^2}{2n^2}}\zeta^r}. 
\end{equation}
On one hand, note that for $(x_1,...,x_n)$ as in the summation indices of \eqref{defhr}, the sum $x_1+....+x_n$ is always a multiple of $n$, as $n$ is odd. In particular, $h_r=0$ if $r$ is not a multiple of $n$. \\
On the other hand, we note that $h_r$ depends only on the class of $r \Mod{n^2}$ by the change of indices $(x_1,...,x_n)\mapsto (x_1+n,...,x_n+n)$. As $n$ is odd, the change of indices $(x_2,...,x_n,x_1)\mapsto (x_1+1,...,x_{n-1}+1,x_n+1)$ implies $h_r=h_{r+n}$. Consequently, \eqref{wronhr} becomes
\begin{equation}
(\Wron_z\theta_{n})(\tau,z)=h_0(\tau)\left(\sum_{r\in \bZ}{q^{\frac{r^2}{2}}\zeta^{rn}}\right)=h_0(\tau)\theta(\tau,nz). \nonumber
\end{equation}
By \eqref{modularityhr} and Lemma \ref{theta}, we find that $h_0$ satisfies a modular invariance property:
\begin{align*}
(\Wron_z\theta_{n})(\tau,z) &=(\det \fu_n(\sigma))^{-1}(\Wron_z\theta_{n})|_{\frac{n^2}{2}}\left[\sigma,j^{\frac{n(n-1)}{2}}\varphi^n\right](\tau,z) \\
&=(\det \fu_n(\sigma))^{-1}(h_0|\left[\sigma,j^{\frac{n(n-1)}{2}}\varphi^{n-1}\right])(\tau) (\theta|_{1/2}[\sigma,\varphi])(\tau,nz) \\
&= \frac{\fu(\sigma)}{\det \fu_n(\sigma))} (h_0|\left[\sigma,j^{\frac{n(n-1)}{2}}\varphi^{n-1}\right])(\tau) \theta(\tau,nz),
\end{align*}
from which one deduces
\begin{equation}
h_0|\left[\sigma,j^{\frac{n(n-1)}{2}}\varphi^{n-1}\right]=\frac{\det \fu_n(\sigma)}{\fu(\sigma)} h_0. \nonumber
\end{equation}
In particular, $h_0$ behaves like a modular form of weight $(n^2-1)/2$ for $\Gamma$ with some character $\sigma\mapsto (\det \fu_n(\sigma))\fu(\sigma)^{-1}$ of norm $1$. From the Fourier expansion \eqref{defhr} of $h_0$, the latter still holds if we replace $\Gamma$ by $\operatorname{SL}_2(\bZ)$ as $h_0$ is invariant by $\tau\mapsto \tau+1$ up to the multiplication by a root of unity. Besides, its order of vanishing at the cusp infinity of $\operatorname{SL}_2(\bZ)$ is at least
\begin{equation}
\frac{1}{2n}\left(1^2+...+\left(\frac{n-1}{2}\right)^2+\left(\frac{-n+1}{2}\right)^2+...+(-1)^2\right)=\frac{n^2-1}{24}. \nonumber
\end{equation}
Consequently, $h_0/\eta^{n^2-1}$ is invariant of weight $0$ for $\operatorname{SL}_2(\bZ)$ (for some character of norm $1$) and bounded on $\cH$. Therefore, it is a constant function on $\cH$. Identifying the constant to be $1!2!\cdots (n-1)!$ as the first nonzero Fourier coefficient of $h_0$ finishes the proof of Theorem \ref{macdoextend}. \\

\paragraph{Acknowledgments:} The author wishes to thank Ken Ono and Antun Milas for their support and interest in this note.


\begin{thebibliography}{}

\bibitem{dyson}
Freeman J. Dyson, Missed opportunities, Bull. Amer. Math. Soc., 78, p.635--652 (1972)


\bibitem{eichlerzagier}
Martin Eichler and Don Zagier, The theory of Jacobi forms, volume 55 of Progress in Mathematics, Birkhäuser Boston, Inc., Boston, MA (1985)

\bibitem{macdonald}
Ian G. Macdonald, Affine root systems and Dedekind’s $\eta$-function, Invent. Math., 15, p.91--143 (1972)

\bibitem{milas2}
Antun Milas, Virasoro algebra, Dedekind $\eta$-function, and specialized Macdonald identities, Transform. Groups, 9(3), p.273--288 (2004)

\bibitem{milas3}
Antun Milas, On certain automorphic forms associated to rational vertex operator algebras, In Moonshine: the first quarter century and beyond, volume 372 of London Math. Soc. Lecture Note Ser., p. 330--357. Cambridge Univ. Press, Cambridge (2010)

\end{thebibliography}
\end{document}